\documentclass[12pt]{article}%
\usepackage{amsmath}
\usepackage{amsfonts}
\usepackage{amssymb}
\usepackage{graphicx}
\usepackage{color}%
\providecommand{\U}[1]{\protect\rule{.1in}{.1in}}
\newtheorem{theorem}{Theorem}

\newtheorem{lemma}[theorem]{Lemma}

\newenvironment{proof}[1][Proof]{\noindent\textbf{#1.} }{\ \rule{0.5em}{0.5em}}
\begin{document}

\title{Some existence results to the Dirichlet problem for the minimal hypersurface
equation on non mean convex domains of a Riemannian manifold}
\author{Ar\'{\i} Aiolfi
\and Jaime Ripoll
\and Marc Soret}
\date{}
\maketitle

\section{Introduction}

\qquad As it is well known, the Dirichlet problem\ for the minimal
hypersurface equation
\begin{equation}
\left\{
\begin{array}
[c]{l}%
\mathcal{M}\left[  u\right]  :=\operatorname{div}\dfrac{\operatorname{grad}%
u}{\sqrt{1+\left\vert \operatorname{grad}u\right\vert ^{2}}}=0\text{ in
}\Omega,\text{ }u\in C^{2,\alpha}\left(  \overline{\Omega}\right) \\
u|_{\partial\Omega}=\varphi,
\end{array}
\right.  \label{Dir}%
\end{equation}
in a bounded domain $\Omega~\subset~\mathbb{R}^{n}$ is solvable for an
arbitrary continuous boundary data $\varphi$ only if the domain is
\textit{mean convex }(Theorem 1 of \cite{JS}). This result (the existence
part)\ has been extended and generalized to Riemannian manifolds (more
generally to constant mean curvature graphs in warped products) in \cite{DHL}.

In the paper \cite{JS} H. Jenkins and J. Serrin noted that a condition
involving $\operatorname*{osc}\left(  \varphi\right)  :=\sup\varphi
-\inf\varphi,$ $\left\vert D\varphi\right\vert $ and $\left\vert D^{2}%
\varphi\right\vert $ should be enough to ensure the solvability of (\ref{Dir})
in arbitrary bounded domains$.$ In fact, they proved that if
$\operatorname*{osc}\left(  \varphi\right)  \leq\mathcal{B}\left(  \left\vert
D\varphi\right\vert ,\left\vert D^{2}\varphi\right\vert ,\Omega\right)  $ then
the (\ref{Dir}) is solvable (Theorem 2 of \cite{JS}). The function
$\mathcal{B}$ has an explicit form (Section 3, p. 179 of \cite{JS} ) and is
infinity at the points where the domain is mean convex. Theorem 2 of \cite{JS}
then extends Theorem 1 of \cite{JS}.

In the present paper we first obtain an extension of Theorem 2 of
Jenkins-Serrin \cite{JS} to the minimal hypersurface PDE\ on a domain $\Omega$
in an arbitrary complete Riemannian manifold $M$. In the next result
$\operatorname{grad}$ and $\operatorname{div}$ are the gradient and divergence
operators in $M.$ Then, $u$ is a solution of (\ref{Dir}) if and only if the
graph of $u$ in $M\times\mathbb{R}$ is a minimal surface. We prove

\begin{theorem}
\label{JS}Let $M^{n}$ be a complete $n-$dimensional Riemannian manifold $M$,
$n\geq2.$ Given a bounded domain $\Omega$ in $M$ - whose boundary
$\partial\Omega$ has mean curvature $H$ w.r.t. the inward unit normal vector -
let $\varphi\in C^{2,\alpha}\left(  \partial\Omega\right)  $ be such that
$\operatorname*{osc}\left(  \varphi\right)  \leq\mathcal{C}\left(  \left\vert
D\varphi\right\vert ,\left\vert D^{2}\varphi\right\vert ,|A|
,\operatorname*{Ric}_{M}\right)  $. Then the Dirichlet problem (\ref{Dir}) for
the minimal hypersurface equation is solvable. Moreover, $\mathcal{C}$ is
given explicitly by (\ref{ourC}) and (\ref{ob})\ and $\mathcal{C}=+\infty$ at
the mean convex points of $\partial\Omega$. It follows that if $\partial
\Omega$ is mean convex (that is, $H\geq0)$ then (\ref{Dir}) is solvable for
any continuous boundary data.
\end{theorem}

Next we apply Theorem \ref{JS} to the exterior Dirichlet problem for the
minimal hypersurface equation when $M$ is complete and noncompact. This
problem consists in proving existence, uniqueness and describing the
asymptotic behavior of a solution of (\ref{Dir}) where $\Omega$ is an exterior
open subset of $M$, that is, $M\backslash\Omega$ is compact.

It seems that the first mathematician to take up with the exterior Dirichlet
problem was J. C. C. Nitsche, who proved (\cite{N}, \S 760) that any solution
$u$ of (\ref{Dir}), in the case $M=\mathbb{R}^{2}$, has at most linear growth
and has a well defined Gauss map at infinity. This problem has been
investigated further by R. Krust \cite{Kru}, E. Kuwert \cite{Kuw}, Kutev and
Tomi in \cite{KT} and, more recently in \cite{RT}.

An investigation of the exterior Dirichlet problem for the minimal surface
equation in the Riemannian setting was initiated in \cite{ER}. There the
authors considered only the special case of vanishing boundary data assuming,
in the case $K_{M}\geq0,$ a condition on the decay of the sectional curvature
of $M$ and, in the case that $K_{M}\leq-k^{2},$ $k>0,$ that $M\backslash
\Omega$ is convex and $M$ simply connected.

In this paper we continue the investigation of \cite{ER} in the case
$K_{M}\leq-k^{2},$ $k>0.$ First, we allow $\Omega$ to be any exterior domain
and the boundary data not need be zero. Moreover, the asymptotic behaviour of
the solution will be prescribed by any given continuous function defined on
the asymptotic boundary $\partial_{\infty}M$ of $M.$ For this last part we use
the ideas of \cite{RTe}, as follows.

Recall that if $M$ is a Hadamard manifold (complete, simply connected,
$K_{M}<0)$ the asymptotic boundary $\partial_{\infty}M$ and the cone
compactifaction $\overline{M}$ of $M$ are well defined by using the so called
cone topology (see \cite{Ch}). According to \cite{RTe}, we say that $M$
satisfies the \emph{strict convexity condition (SC condition)} if, given
$x\in\partial_{\infty}M$ and a relatively open subset $W\subset\partial
_{\infty}M$ containing $x,$ there exists a $C^{2}$ open subset $\Omega
\subset\overline{M}$ such that $x\in\operatorname*{Int}\left(  \partial
_{\infty}\Omega\right)  \subset W,$ where $\operatorname*{Int}\left(
\partial_{\infty}\Omega\right)  $ denotes the interior of $\partial_{\infty
}\Omega$ in $\partial_{\infty}M,$ and $M\setminus\Omega$ is convex. We prove

\begin{theorem}
\label{SC}Let $M$ be a Hadamard manifold satisfying the SC condition and
assume that $K_{M}\leq-k^{2},$ $k>0$. Let $\Omega$ be an exterior
$C^{2,\alpha}$ domain. Given $\varphi\in C^{2,\alpha}\left(  \partial
\Omega\right)  $ such that $\operatorname*{osc}\left(  \varphi\right)
\leq\mathcal{C}\left(  \left\vert D\varphi\right\vert ,\left\vert D^{2}%
\varphi\right\vert ,|A|, Ric_{M}\right)  $ and $\psi\in C^{0}\left(
\partial_{\infty}M\right)  $ there is an unique solution $u\in C^{2,\alpha
}\left(  \overline{\Omega}\right)  $ of $\mathcal{M}\left[  u\right]  =0$ in
$\Omega$ such that $u|_{\partial\Omega}=\varphi.$ Moreover, $u$ extends
continuously to $\partial_{\infty}M$ and $u|_{\partial_{\infty}M}=\psi.$
\end{theorem}

We mention that under the hypothesis $K_{M}\leq-k^{2}<0$, any $2-$dimensional
Hadamard manifold satisfies the SC convexity condition, since any two distinct
points of $\partial_{\infty}M$ can be connected by a minimizing geodesic. It
is also proved in \cite{RTe} that this condition is also satisfied if the
metric of $M$ is rotationally symmetric or if the sectional curvature of $M$
has at most exponential decay, precisely, if $\inf_{B_{R}}K_{M}\geq
-Ce^{(k-\varepsilon)R}$ for $R\geq R_{0}$ and for some $\varepsilon>0.$ In
\cite{HR} it is proved that if the SC condition is not required then there are
examples of $3-$dimensional complete manifolds with $K_{M}\leq-k^{2}<0$ in
which only the constant functions are bounded solutions of the minimal PDE\ on
$M$ that extends continuously to $\partial_{\infty}M.$ In such manifolds, if
$u\in C^{\infty}\left(  M\right)  \cap C^{0}\left(  \overline{M}\right)  $ is
solution of an exterior Dirichlet problem for the minimal hypersurface
equation then $u|_{\partial_{\infty}M}$ is constant.

\section{An extension of a Theorem of Jenkins and Serrin}

\qquad We begin with some preliminary facts.

\subsection{Normal coordinates for the inner halftube of $\partial\Omega$}

Let $\varphi\in C^{2}\left(  \partial\Omega\right)  $ be given. Let $d$ be the
Riemannian distance in $M$ and set $\rho(z):=d(z,\partial\Omega)$, $z\in
\Omega$. For $\rho_{0}$ sufficiently small, the normal exponential map
\[
\exp_{\partial\Omega}:\partial\Omega\times\lbrack0,\rho_{0})\longrightarrow
U_{\rho_{0}}=\left\{  z\in\Omega;\rho(z)<\rho_{0}\right\}  \subset M,
\]
is a diffeomorphism.\newline Let $\left\{  T_{1}(x),...,T_{n-1}(x)\right\}
_{x\in V_{r_{0}}}$ be the orthonormal frame defined on an neighborhood
$V_{r_{0}}\subset\partial\Omega$ of a point $p\in\partial\Omega$ by parallel
transport of a given orthonormal frame at $p$ (by construction $\nabla_{T_{i}%
}T_{j}(p)=0,i,j=1,\cdots,n-1$). \newline For each $i\in\left\{
1,...,n-1\right\}  $, extend also $T_{i}|_{\partial\Omega}$ by parallel
transport along the normal geodesic $t\rightarrow\exp_{\partial\Omega}t\eta$,
where $\eta$ is the inward unitary normal field to $\partial\Omega$ and
$t\leq\rho_{0}$, and denote this extension again by $T_{i}$. Then, for each
$(x,t)\in V_{r_{0}}\times\left[  0,\rho_{0}\right]  $, $\left\{
T_{1}(x,t),...,T_{n-1}(x,t)\right\}  $ is a orthonormal frame on the
equidistant hypersurface $\rho\left(  z\right)  =t$. We complete the
orthonormal frame on $U_{\rho_{0}}$ by setting $T_{n}(x,t)=\nabla\rho(x,t)$
for all $(x,t)\in V_{r_{0}}\times\lbrack0,\rho_{0})$. We extend also $\varphi$
to $U_{\rho_{0}}$ by setting $\varphi(\exp_{\partial\Omega}(x,t))=\varphi(x),$
$(x,t)\in\partial\Omega\times\lbrack0,\rho_{0}).$\newline By construction
$\nabla_{T_{n}}T_{i}(x,t)=0$ where $(x,t)\in V_{r_{0}}\times\lbrack0,\rho
_{0})$, $i\in\left\{  1,...,n-1\right\}  $ and $T_{n}(\varphi)=\nabla_{T_{n}%
}T_{n}=0$. Define
\begin{equation}
\omega(z)=\varphi(z)+\psi(\rho(z)) , \label{omega2}%
\end{equation}
where $z\in U_{\rho_{0}}$ and \emph{\ }$\psi\in C^{2}\left(  \left[
0,\infty\right)  \right)  $ is to be determined later. Let $\mathcal{M}$
denote, as above, the minimal hypersurface equation operator. We have%

\begin{equation}
W^{3}\mathcal{M}\left(  \omega\right)  =-\frac{1}{2}\nabla_{\nabla\omega
}(|\nabla\omega)|^{2})+W^{2}\operatorname{div}(\nabla\omega) \label{mse}%
\end{equation}
with $W^{2}:=\left(  1+\left\vert \nabla\omega\right\vert ^{2}\right)  $. Then
$\mathcal{M}\left(  \omega\right)  \leq0$ if only if%
\begin{equation}
-\sum\limits_{i=1}^{n}\left\langle \nabla_{T_{i}}\nabla\omega,\nabla
\omega\right\rangle T_{i}\left(  \omega\right)  +W^{2}\sum\limits_{i=1}%
^{n}\left\langle \nabla_{T_{i}}\nabla\omega,T_{i}\right\rangle \leq0.
\label{a0}%
\end{equation}
\newpage

\begin{lemma}
\label{ES} The following equalities hold for $\omega$ :
\[
\left\{
\begin{array}
[c]{ll}%
\left\langle \nabla_{T_{n}}\nabla\omega,\nabla\omega\right\rangle
=\sum_{i,k=1}^{n-1}II(T_{i},T_{k})T_{i}(\varphi)T_{k}\left(  \varphi\right)
+\psi^{\prime}\psi^{\prime\prime} & \\
\left\langle \nabla_{T_{n}}\nabla\omega,T_{n}\right\rangle =\psi^{\prime
\prime} & \\
\sum\limits_{i=1}^{n-1}\left\langle \nabla_{T_{i}}\nabla\omega,\nabla
\omega\right\rangle T_{i}\left(  \varphi\right)  =\sum\limits_{i,j=1}%
^{n-1}T_{i}\left(  T_{j}\left(  \varphi\right)  \right)  T_{j}\left(
\varphi\right)  T_{i}\left(  \varphi\right)  & \\
\sum\limits_{i=1}^{n-1}\left\langle \nabla_{T_{i}}\nabla\omega,T_{i}%
\right\rangle =\sum\limits_{i=1}^{n-1}T_{i}\left(  T_{i}\left(  \varphi
\right)  \right)  +\sum\limits_{i=1}^{n-1}T_{i}\left(  \varphi\right)
\operatorname{div}_{\Omega(\rho)}T_{i}-\psi^{\prime}(n-1)H. &
\end{array}
\right.
\]

\end{lemma}

\begin{proof}
We will use throughout the proof that $\nabla_{T_{n}}T_{i}=0,$ $i=1,\cdots,n$.
To prove equality one and two we need to compute $\nabla_{T_{n}}\nabla\omega$.
Since
\begin{equation}
\nabla\omega=\sum\limits_{i=1}^{n-1}T_{i}\left(  \varphi\right)  T_{i}%
+\psi^{\prime}T_{n} \label{eeg}%
\end{equation}
then%
\begin{equation}
\nabla_{T_{n}}\nabla\omega=\sum\limits_{i=1}^{n-1}T_{n}(T_{i}\left(
\varphi\right)  )T_{i}+T_{n}(\psi^{\prime})T_{n}. \label{bc}%
\end{equation}
We obtain for the first terms of (\ref{bc})%
\begin{align}
T_{n}\left(  T_{i}(\varphi)\right)   &  =[T_{n},T_{i}]\varphi=(\nabla_{T_{n}%
}T_{i}-\nabla_{T_{i}}T_{n})\varphi=-\nabla_{T_{i}}T_{n}\left(  \varphi\right)
\nonumber\\
&  =-\sum\limits_{k=1}^{n-1}\left\langle \nabla_{T_{i}}T_{n},T_{k}%
\right\rangle T_{k}\left(  \varphi\right)  =\sum_{k}II(T_{i},T_{k})T_{k}(\phi)
\label{ps1}%
\end{align}
\newline Moreover%
\begin{equation}
T_{n}\left(  \psi^{\prime}\left(  \rho\right)  \right)  T_{n}=\psi
^{\prime\prime}\left(  \rho\right)  T_{n}\left(  \rho\right)  T_{n}%
=\psi^{\prime\prime}T_{n} \label{ps2}%
\end{equation}
\newline\textbf{Proof of $i)$ and $ii)$:} From (\ref{eeg}), (\ref{bc}),
(\ref{ps1}) and (\ref{ps2}) we obtain\textbf{ }%
\[
\left\langle \nabla_{T_{n}}(\nabla\omega),\nabla\omega\right\rangle
=\sum_{i=1}^{n-1}T_{i}(\varphi)T_{n}(T_{i}(\varphi))+\psi^{\prime}T_{n}%
(\psi^{\prime})=\sum_{i,k=1}^{n-1}II(T_{i},T_{k})T_{i}(\varphi)T_{k}\left(
\varphi\right)  +\psi^{\prime}\psi^{\prime\prime}%
\]
and $\left\langle \nabla_{Tn}\nabla\omega,T_{n}\right\rangle =\psi
^{\prime\prime}$.\newline\textbf{Proof of $iii)$}: Note that $\nabla
\varphi^{T}$ - the projection of $\nabla\varphi$ on hypersurfaces parallel to
$\partial\Omega$ - is $\nabla\varphi$ since $\varphi$ is independent of $\rho
$. Furthermore $\nabla_{T_{i}}\psi=0,i=1,\cdots n-1$; hence we have
\[
\frac{1}{2}(\nabla\varphi)^{T}(|\nabla\omega|^{2})=\sum_{i=1}^{n-1}%
T_{i}(\varphi)\nabla_{\nabla\varphi^{T}}(T_{i}(\varphi))=\sum\limits_{i,j=1}%
^{n-1}T_{i}\left(  T_{j}\left(  \varphi\right)  \right)  T_{j}\left(
\varphi\right)  T_{i}\left(  \varphi\right)  .
\]
\textbf{Proof of $iv)$}: Using (\ref{eeg}) we have
\[
\sum_{i=1}^{n-1}\left\langle \nabla_{T_{i}}\nabla\omega,T_{i}\right\rangle
=\sum_{i=1}^{n-1}T_{i}(T_{i}(\varphi))-(n-1)\psi^{\prime}H+\sum_{i=1}%
^{n-1}T_{i}\left(  \varphi\right)  \operatorname{div}_{\Omega(\rho)}(T_{i}).
\]

\end{proof}

\subsubsection{Barriers for the Dirichlet problem on $\Omega$ for the minimal
surface equation}

\begin{lemma}
\label{Bar} Let $H$ and $A$ respectively the mean curvature and the shape
operator of $\partial\Omega$ w.r.t. to the inner orientation and set $H_{\inf
}:=\inf_{\partial\Omega}H$. Let $R$ be an upperbound of the Ricci curvature of
$M$. The function\newline$\omega\left(  z\right)  :=\varphi\left(  z\right)
+a\ln\left(  1+bt\right)  $ is superharmonic w.r.t. $\mathcal{M}$ on
$U_{\varepsilon}$, where:\newline i) if $H_{\inf}<0$, $a=b^{-1}$,
$0<\varepsilon<\min\{\frac{1}{2b},\rho_{0}\}$ being the constant $b$ given by
\begin{equation}%
\begin{array}
[c]{c}%
b/3=||D\varphi||_{\partial\Omega}^{2}\left(  ||D^{2}\varphi||_{\partial\Omega
}+||A||_{\partial\Omega}\right)  \\
+(2+||D\varphi||_{\partial\Omega}^{2})(||D^{2}\varphi||_{\partial\Omega
}+(n-1)^{2}\rho_{0}\left\vert \left\vert D\varphi\right\vert \right\vert
_{\partial\Omega}R-(n-1)H_{\inf});
\end{array}
\label{ob}%
\end{equation}
\newline ii) if $H_{\inf}\geq0$, $b>a^{-1}$, $0<\varepsilon<\min
\{a-b^{-1},\rho_{0}\}$, being the constant $a$ given by%
\[%
\begin{array}
[c]{c}%
a^{-1}=||D\varphi||_{\partial\Omega}^{2}\left(  ||D^{2}\varphi||_{\partial
\Omega}+||A||_{\partial\Omega}\right)  \\
+(2+||D\varphi||_{\partial\Omega}^{2})(||D^{2}\varphi||_{\partial\Omega
}+(n-1)^{2}\rho_{0}\left\vert \left\vert D\varphi\right\vert \right\vert
_{\partial\Omega}R).
\end{array}
\]

\end{lemma}

\begin{proof}
Let us introduce the following notations for terms containing only $\varphi$
and its derivatives:
\[
\alpha:=\sum\limits_{i=1}^{n-1}\left[  T_{i}\left(  \varphi\right)  \right]
^{2},\beta=\sum\limits_{i=1}^{n-1}T_{i}\left(  T_{i}\left(  \varphi\right)
\right)  ,\ \mu:=\sum\limits_{j=1}^{n-1}T_{j}\left(  \varphi\right)
\operatorname{div}T_{j}%
\]%
\[
\lambda:=\sum\limits_{i,j=1}^{n-1}T_{i}\left(  T_{j}\left(  \varphi\right)
\right)  T_{j}\left(  \varphi\right)  T_{i}\left(  \varphi\right)
,\ \theta:=\sum_{i,k=1}^{n-1}II(T_{i},T_{k})T_{i}(\varphi)T_{k}\left(
\varphi\right)  .
\]
From Lemma \ref{ES}, we plug in inequality \eqref{a0} the preceeding terms and
obtain a differential inequality for $\psi$:
\[
\label{inegalite1}-\lambda+\left(  1+\alpha+\left[  \psi^{\prime}\right]
^{2}\right)  \left(  \sigma-\psi^{\prime}(n-1)H\right)  \newline-\left(
\theta+\psi^{\prime}\psi^{\prime\prime}\right)  \psi^{\prime}+\newline\left(
1+\alpha+\left[  \psi^{\prime}\right]  ^{2}\right)  \psi^{\prime\prime}\leq0.
\]
where $\sigma:=\beta+\mu$.\newline Set
\begin{equation}
H_{\inf}:=\inf_{\partial\Omega}H|_{\partial\Omega}\label{hinf}%
\end{equation}
and define $\psi\left(  t\right)  =a\ln\left(  1+bt\right)  \text{,}$ where
$a>0$ and $b>0$ are constant to be determined. Since $\psi^{\prime}>0$, we can
replace the function $H$ in the inequality \eqref{inegalite1} by $H_{\inf}$ :
\[%
\begin{array}
[c]{cc}%
-\lambda+\left(  1+\alpha\right)  \sigma+\sigma\left[  \psi^{\prime}\right]
^{2}-\theta\psi^{\prime}+\left(  1+\alpha\right)  \psi^{\prime\prime} & +\\
-\left(  n-1\right)  H_{\inf}\left[  \psi^{\prime}\right]  ^{3}-\left(
n-1\right)  \left(  1+\alpha\right)  H_{\inf}\psi^{\prime} & \leq0.
\end{array}
\]
\indent Setting $\delta:=-\lambda+\sigma\left(  1+\alpha\right)  \text{,
}\ c:=\left(  n-1\right)  \left(  1+\alpha\right)  >0$, then last inequality
becomes%
\begin{equation}%
\begin{array}
[c]{cc}%
\delta+\sigma\left[  \psi^{\prime}\right]  ^{2}-\theta\psi^{\prime}+\left(
1+\alpha\right)  \psi^{\prime\prime} & +\\
-\left(  n-1\right)  H_{\inf}\left[  \psi^{\prime}\right]  ^{3}-cH_{\inf}%
\psi^{\prime} & \leq0.
\end{array}
\label{Ar1}%
\end{equation}
\indent We first suppose $H_{\inf}<0$. Set $ab=1$; then replacing
$\psi^{\prime}=\left(  1+bt\right)  ^{-1}$ and $\psi^{\prime\prime}=-\left[
\psi^{\prime}\right]  ^{2}b$,
\[%
\begin{array}
[c]{cc}%
\delta\left(  1+bt\right)  ^{3}+\sigma\left(  1+bt\right)  -\theta\left(
1+bt\right)  ^{2}-\left(  1+\alpha\right)  b\left(  1+bt\right)   & \\
-\left(  n-1\right)  H_{\inf}-cH_{\inf}\left(  1+bt\right)  ^{2}%
\leq0.\label{uae} &
\end{array}
\]
Taking absolute values, dividing by $1+bt$ and expanding w.r.t. $t$,
(\ref{uae}) is true if%
\begin{equation}%
\begin{array}
[c]{cc}%
|\delta|b^{2}t^{2}+\left(  2\left\vert \delta\right\vert +\left\vert
\theta\right\vert -cH_{\inf}\right)  bt+\left\vert \delta\right\vert
+\left\vert \theta\right\vert +\left\vert \sigma\right\vert -b\left(
\alpha+1\right)   & +\\
-(c+n-1)H_{\inf} & \leq0
\end{array}
\label{bom}%
\end{equation}
It is clear that a sufficient condition for inequality (\ref{bom}) is that :
\begin{equation}
\left\{
\begin{array}
[c]{l}%
t\leq\frac{1}{\sqrt{3\left\vert \delta\right\vert b}}\\
t\leq\frac{1}{3\left[  2\left\vert \delta\right\vert +\left\vert
\theta\right\vert -cH_{\inf}\right]  }\\
b/3\geq\left\vert \delta\right\vert +\left\vert \sigma\right\vert +\left\vert
\theta\right\vert -\left(  c+n-1\right)  H_{\inf}%
\end{array}
\right.  .\label{si}%
\end{equation}
Notice that these inequalities are \textit{a fortiori} satisfied if we replace
in these expressions the functions $\alpha$, $\left\vert \beta\right\vert $,
$\left\vert \mu\right\vert $, $\left\vert \lambda\right\vert $ and $\left\vert
\theta\right\vert $ by their supremum on $\partial\Omega$. We obtain
\[
\alpha:=||D\varphi||^{2},\beta=||D^{2}\varphi||,\ \lambda:=||D^{2}%
\varphi||||D\varphi||^{2},\ \theta:=||A||||D\varphi||^{2}.
\]
To estimate $\mu$ we need to bound $\operatorname{div}(T_{j})$. We derivate
the equation $\nabla_{T_{n}}T_{j}=0$ with respect to $T_{i},i=1,\cdots,n-1$.
We obtain an evolution equation for $\operatorname{div}T_{j}$ along the normal
geodesic : $\nabla_{T_{n}}\operatorname{div}(T_{j})=(n-1)\operatorname*{Ric}%
(T_{n},T_{j})$, with the initial condition $\operatorname{div}(T_{j}(p)=0$.
This yields
\[
\mu:=(n-1)^{2}||D\varphi||\rho_{0}\sup_{y\in V_{r_{0}}\times\lbrack0,\rho
_{0}]}\operatorname*{Ric}(y).
\]
\newline Let us fix $b$ such that third inequality in (\ref{si}) is an
equality. We then obtain expression of $b$ in Lemma \ref{Bar}. Replacing $b$
by its expression (\ref{ob}) the first two inequalities of \eqref{si} hold if
\[
t\leq t_{0}:=\min\left(  {\frac{1}{2b},\rho_{0}}\right)  .
\]
This conclude the proof of i).

Now, suppose $H_{\inf}\geq0$. In this case, inequality (\ref{Ar1}) is
satisfied if\newline $\delta+\sigma\left(  \psi^{\prime}\right)  ^{2}%
-\theta\psi^{\prime}+\left(  1+\alpha\right)  \psi^{\prime\prime}\leq0$ which,
after replacing $\psi^{\prime}=ab\left(  1+bt\right)  ^{-1}$ and $\psi
^{\prime\prime}=-\left[  \psi^{\prime}\right]  ^{2}a^{-1}$, become%
\begin{equation}
\delta\left(  1+bt\right)  ^{2}+\sigma a^{2}b^{2}-\theta ab\left(
1+bt\right)  -\left(  1+\alpha\right)  ab^{2}\leq0.\label{e0}%
\end{equation}
Since $\alpha\geq0$, (\ref{e0}) is satisfied if%
\begin{equation}
\left\vert \delta\right\vert \left(  1+bt\right)  ^{2}+\left\vert
\sigma\right\vert a^{2}b^{2}+\left\vert \theta\right\vert ab\left(
1+bt\right)  \leq ab^{2}.\label{e1}%
\end{equation}
Notice that for $1+bt\leq ab$, equation (\ref{e1}) is true if we replace
$1+bt$ by $ab$, obtaining
\[
a\left(  \left\vert \delta\right\vert +\left\vert \sigma\right\vert
+\left\vert \theta\right\vert \right)  \leq1.
\]
Fix $a=1/\left(  \left\vert \delta\right\vert +\left\vert \sigma\right\vert
+\left\vert \theta\right\vert \right)  $ where we already assume for
$\left\vert \delta\right\vert ,\left\vert \sigma\right\vert $ and $\left\vert
\theta\right\vert $ their supremum on $\partial\Omega$. Then, for all $b>1/a$
and $t\leq\min\left\{  a-1/b,\rho_{0}\right\}  $ we have (\ref{e1}) satisfied
and this conclude the proof of ii).
\end{proof}

\subsection{Proof of Theorem \ref{JS}.}

\qquad Let
\begin{equation}
\mathcal{C}=\frac{1}{b}\ln(1+b\varepsilon), \label{ourC}%
\end{equation}
where $b$ and $\varepsilon$ are defined in i) of Lemma \ref{Bar}, if $H_{\inf
}<0$ ($\mathcal{C=+\infty}$ if $H_{\inf}\geq0$ - according with ii) of Lemma
\ref{Bar}). For the first part where $\varphi\in C^{2,\alpha}\left(
\partial\Omega\right)  $ one uses the continuity method by setting
\[
V=\left\{  t\in\left[  0,1\right]  \text{
$\vert$
}\exists u_{t}\in C^{2,\alpha}\left(  \overline{\Omega}\right)  \text{ such
that }\mathcal{M}\left[  u_{t}\right]  =0\text{, }u_{t}|_{\partial\Omega
}=t\varphi\right\}  .
\]
Clearly $V\neq\varnothing$ since $t=0\in V.$ Moreover, $V$ is open by the
implicit function theorem.

Let $t_{n}\in V$ be a sequence converging to $t\in\left[  0,1\right]  $ and
$u_{n}:=u_{t_{n}}\in C^{2,\alpha}\left(  \overline{\Omega}\right)  $ be the
solutions such that $u_{n}|_{\partial\Omega}=t_{n}\varphi.$ By the maximum
principle the sequence $u_{n}$ has uniformly bounded $C^{0}$ norm. Moreover,
$\varphi-\psi|_{\partial\Omega}\leq u_{n}|_{\partial\Omega}\leq\omega
|_{\partial\Omega}$ where the functions $\varphi$ and $\psi$ are defined in
equation \eqref{omega2} and Lemma \ref{Bar}. It follows that%

\[
\max_{\partial\Omega}\left\vert \operatorname{grad}u_{n}\right\vert \leq
\max\left\{  \max_{\partial\Omega}\left\vert \operatorname{grad}%
\sigma\right\vert ,\max_{\partial\Omega}\left\vert \operatorname{grad}%
\omega\right\vert \right\}  <\infty.
\]
By Section 5 of \cite{DHL} there is $C>0$ such that $\max_{\Omega}\left\vert
\operatorname{grad}u_{n}\right\vert \leq C$ so that $\left\vert u_{n}%
\right\vert _{1}\leq D<\infty$ with $D$ not depending on $n.$ H\"{o}lder
estimates and PDE\ linear elliptic theory (\cite{GT}) guarantees that $u_{n}$
is equicontinous in the $C^{2,\beta}$ norm for some $\beta>0$ and hence
contains a subsequence converging uniformly on the $C^{2}$ norm to a solution
$u\in C^{2}\left(  \overline{\Omega}\right)  .$ Regularity theory of linear
elliptic PDE (\cite{GT}) implies that $u\in C^{2,\alpha}\left(  \overline
{\Omega}\right)  .$ This proves the first part of the theorem.

Assume now that $\Omega$ is mean convex and let $\varphi\in C^{0}\left(
\partial\Omega\right)  $ be given. Let $\varphi_{n}^{\pm}\in C^{2,\alpha
}\left(  \partial\Omega\right)  $ be a monotonic sequence of functions
converging from above and from below to $\varphi$ in the $C^{0}$ norm. It
follows by what we have proved above the existence of solutions $u_{n}^{\pm
}\in C^{2,\alpha}\left(  \overline{\Omega}\right)  $ of $\mathcal{M}=0$ in
$\Omega$ such that $u_{n}^{\pm}|_{\partial\Omega}=\psi_{n}^{\pm}.$ The
sequence $u_{n}^{\pm}$ is uniformly bounded in the $C^{0}$ norm by the maximum
principle. Therefore, by Theorem 1.1 of \cite{S} and linear elliptic PDE
theory the sequence $u_{n}^{\pm}$ contains a subsequence $v_{n}\in
C^{2,\alpha}\left(  \overline{\Omega}\right)  $ converging uniformly on the
$C^{2}$ norm on compacts subsets of $\Omega$ to a solution $u\in C^{2}\left(
\Omega\right)  $ of $\mathcal{M}=0$. Since
\[
\varphi_{n}^{-}\leq u_{n}^{-}\leq u_{n}^{+}\leq\varphi_{n}^{+}%
\]
and $\varphi_{n}^{\pm}$ converges to $\varphi$ it follows by the maximum
principle that $u$ extends continously to $\overline{\Omega}$ and
$u|_{\partial\Omega}=\varphi.$ This concludes the proof of Theorem \ref{JS}.

\section{The exterior Dirichlet problem for the \newline minimal hypersurface
PDE with prescribed asymptotic boundary}

\qquad For the proof of Theorem \ref{SC} we shall make use of the following
definition given in \cite{RTe}. We first recall that a function $\Sigma\in
C^{0}\left(  M\right)  $ is a supersolution for $\mathcal{M}$ if, given a
bounded domain $U\subset M$ and if $u\in C^{0}\left(  \overline{U}\right)  $
is a solution of $\mathcal{M}=0$ in $U$, then $u|_{\partial U}\leq
\Sigma|_{\partial U}$ implies that $u\leq\Sigma|_{U}$.

Given $x\in\partial_{\infty}M$ and an open subset $\Omega\subset M$ such that
$x\in\partial_{\infty}\Omega$, an \emph{upper barrier for $\mathcal{M}$
relative to $x$ and $\Omega$ with height $C$} is a function $\Sigma\in
C^{0}(M)$ such that

\begin{description}
\item \textrm{(i)} $\Sigma$ is a supersolution for $\mathcal{M}$;

\item \textrm{(ii)}$\Sigma\geq0$ and $\lim_{p\in M,\,p\rightarrow x}%
\Sigma(p)=0$, w.r.t. the cone topology and according to \cite{Ch};

\item \textrm{(iii)} $\Sigma_{M\setminus\Omega} \ge C$.
\end{description}

Similarly, we define subsolutions and lower barriers.

We say that $M$ is \emph{regular at infinity with respect to }$\mathcal{M}$
if, given $C>0$, $x\in\partial_{\infty}M$ and an open subset $W\subset
\partial_{\infty}M$ with $x\in W$, there exist an open set $\Omega\subset M$
such that $x\in\operatorname*{Int}\partial_{\infty}\Omega\subset W$ and upper
and lower barriers $\Sigma,\sigma:M\rightarrow\mathbb{R}$ relatives to $x$ and
$\Omega$, with height $C$.

\subsection{Proof of Theorem \ref{SC}}

\qquad Consider a continuous extension $\Psi$ of $\psi\in C^{0}\left(
\partial_{\infty}M\right)  $ which is $C^{\infty}$ in $M.$ That is, $\Psi\in
C^{\infty}\left(  M\right)  \cap C^{0}\left(  \overline{M}\right)  ,$
$M:=M\cup\partial_{\infty}M$ and $\Psi|_{\partial_{\infty}M}=\psi.$ Let $o\in
M$ be a fixed point in $M$ and let $N>0$ be such that the open geodesic ball
$B_{n}$ centered at $o$ with radius $n$ contains the boundary $\partial\Omega$
of the exterior domain $\Omega$ ($n\in\mathbb{N}$ and $n\geq N$). Set
$\Omega_{n}=\Omega\cap B_{n}.$ It follows from the Hessian comparison theorem
that $\partial\Omega_{n}\backslash\partial\Omega$ is convex, in particular
mean convex. From this fact and from the hypothesis on the oscillation of
$\varphi$, it follows from Theorem \ref{JS} the existence of a solution
$u_{n}\in C^{2,\alpha}\left(  \overline{\Omega}_{n}\right)  $ of
$\mathcal{M}=0$ in $\Omega_{n}$ such that $u_{n}|_{\partial\Omega}=\varphi$
and $u_{n}|_{\partial\Omega_{n}\backslash\partial\Omega}=\Psi|_{\partial
\Omega_{n}\backslash\partial\Omega}.$

By the maximum principle $u_{n}$ is uniformly bounded. It follows from Theorem
1.1 of \cite{S} and the diagonal method that $u_{n}$ contains a subsequence
converging uniformly on the $C^{2}$ norm on compact subsets of $\Omega$ to a
solution $u\in C^{\infty}\left(  \Omega\right)  $ of $\mathcal{M}=0.$ By Lemma
\ref{lem} and regularity theory $u\in C^{2,\alpha}\left(  \overline{\Omega
}\right)  $ and $u|_{\partial\Omega}=\varphi.$

Since $M$ satisfies the SC condition- given in the Introduction- it follows
from Theorem 10 of \cite{RTe} that $M$ is regular at infinity with respect to
the minimal hypersurface PDE. From Theorem 4 of \cite{RTe} it follows that $u$
extends continuously to $\partial_{\infty}M$ and satisfies the boundary
condition $u|_{\partial_{\infty}M}=\psi.$ This concludes the proof of Theorem
\ref{SC}.


\bigskip

\begin{thebibliography}{99}                                                                                               %


\bibitem {Ch}H. I. Choi: \textit{Asymptotic Dirichlet problems for harmonic
functions on Riemannian manifolds}. Transaction of the AMS, 281, 2, (1984), 691--716

\bibitem {DHL}M. Dajczer, P. Hinojosa, J.H. de Lira. \textit{Killing graphs
with prescribed mean curvature,} Calc. Var. Partial Differential Equations 33
(2008) 231--248

\bibitem {ER}N. do Espirito-Santo, J. Ripoll: \textit{Some existence results
on the exterior Dirichlet problem for the minimal hypersurface equation, }Ann.
I. H. Poincar\'{e}/An Non Lin, 28 (2011), 385--393

\bibitem {GT}D. Gilbarg and N. S. Trudinger, \textit{Elliptic partial
differential equations of second Order}, Springer-Verlag (1998), Berlin

\bibitem {HR}I. Holopainen, J. Ripoll: \textit{\textquotedblleft Nonsolvabily
of the asymptotic Dirichlet for some quasilinear elliptic PDE on Hadamard
manifolds\textquotedblright}, work in progress

\bibitem {JS}H. Jenkins, J. Serrin: \textit{The Dirichlet problem for the
minimal surface equation in higher dimensions, }J. Reine Angew. Math. 229
(1968), 170--187

\bibitem {Kru}R. Krust: \textit{Remarques sur le probl\`{e}me exterior de
Plateau, }Duke Math. Jr, Vol 59 (1989), 161-173

\bibitem {KT}N. Kutev, F. Tomi: \textit{Existence and Nonexistence for the
exterior Dirichlet problem for the minimal surface equation in the plane,
}Differential and Integral Equ., Vol 11, N. 6 (1998), 917-928

\bibitem {Kuw}E. Kuwert: \textit{On solutions of the exterior Dirichlet
problem for the minimal surface equation, }Ann. I. H. Poincar\'{e}/An Non Lin,
10 (1993), 445--451

\bibitem {N}J. C. C. Nitsche: \textit{Vorlesungen \"{u}ber Minimalfl\"{a}chen,
}Grundlehren der Math. Wiss., Vol 199 (1975), Springer

\bibitem {RT}J. Ripoll, F. Tomi: \textit{On solutions to the exterior
Dirichlet problem for the minimal surface equation with catenoidal ends, }to
appear in the Advances of Calculus of Variations, 2013

\bibitem {RTe}J. Ripoll, M. Telichevesky: \textit{Regularity at infinity of
Hadamard manifolds with respect to some elliptic operators and applications to
asymptotic Dirichlet problems,} to appear in the Transaction of the AMS (2013)

\bibitem {S}J. Spruck,\textquotedblleft\emph{Interior gradient estimates and
existence theorems for constant mean curvature graphs in $M^{n}\times
\mathbb{R}$}\textquotedblright, Pure and Applied Mathematics Quarterly
\textbf{3 (3)}(Special Issue: In honor of Leon Simon, Part 1 of 2): 785--800, 2007
\end{thebibliography}
\end{document}